\documentclass[11pt]{article}
\usepackage[T1]{fontenc}
\usepackage{lmodern}
\usepackage{mathtools,microtype,needspace}
\usepackage{authblk}

\usepackage[hidelinks]{hyperref}
\usepackage{max2024}

\newtheorem{theorem}{Theorem}
\newtheorem{lemma}[theorem]{Lemma}

\newtheorem{corollary}[theorem]{Corollary}
\newtheorem{question}{Question}
\newcommand{\Cov}{\operatorname{Cov}}
\newcommand{\TV}{d_{\mathrm{TV}}}
\newcommand{\wsp}{\hspace{24pt}}
\newcommand{\imrsp}{\hsp\imr\hsp}
\renewcommand{\lq}{\prec}

\providecommand{\sd}{\operatorname{sd}}

\newcommand{\cpnorm}[1]{\left\lVert#1\right\rVert}

\title{Proof of the Kahn–Saks Conjecture}
\author{Max Aires\thanks{Department of Mathematics, University of Southern California}}

\begin{document}
\maketitle
\begin{abstract}
Let $\PP(x\lq y)$ be the probability that $x$ precedes $y$ in a uniformly random linear extension of an $n$-element poset $P$, and define the balancing coefficient to be $\de(x,y)=\min(\PP(x\lq y),\PP(y\lq x))$ with $\de(P)=\max_{x,y}\de(x,y)$. Building on work in~\cite{AK}, we prove (Theorem~\ref{cor:kahn-saks}) that sufficiently large width forces $\de(P)$ to be arbitrarily close to $1/2$, answering a long-standing conjecture of Kahn and Saks~\cite{KS}. In fact, we prove the stronger result (Theorem~\ref{thm:main}) that large width forces one of two configurations in our poset: either a nearly uniform order on $k$ vertices, or an almost fixed order on $t$ vertices with one further vertex inserted uniformly among the $t+1$ slots. We also show that, for fixed $k$, the first possibility must occur within any antichain $X$ of size $\Om(n^{2/3})$.
\end{abstract}

\section{Introduction}
Let $P$ be a partially ordered set; we shall abuse notation by also using $P$ for its ground set. A \em{linear extension} of $P$ is a linear order of the elements of $P$ which agrees with the poset relations. See~\cite{CP} for a good survey on linear extensions.

Much effort over the years has been dedicated to giving lower bounds on $\de(P)$, and in particular to proving the famed ``1/3-2/3 Conjecture'' that $\de(P)\ge 1/3$ whenever $P$ is not a chain, posed independently in~\cite{Kis68} and~\cite{Fre75}; see~\cite{KS},\cite{Lin},\cite{BFT},\cite{ACPP} for a (non-exhaustive) sampling of partial progress. Our focus here is instead on conditions which force $\de(P)$ to be close to $1/2$, i.e. that force the existence of a pair of elements which appear in either order with roughly equal probability.

Intuitively, such elements should almost always exist for a ``typical'' poset, and a variety of general, sufficient conditions have been proposed. Most notably, in 1984, Kahn and Saks~\cite{KS} conjectured that, if $w(P)$ is the width of $P$,
$$w(P)\to\oo\imrsp\de(P)\to1/2.$$
Here and below all limits are along sequences of posets with $n=|P|\to\oo$.  Until recently, the best result in this direction was by Koml\'os, who proved~\cite{Komlos} in 1990 that $\de(P)\to 1/2$ under the stronger hypothesis that $|\min(P)|=\Om(n)$; in particular, by a result of Kleitman and Rothschild \cite{KR} this implies that $\de(P)\to 1/2$ asymptotically almost surely for a uniformly random poset on $[n]$ (see also \cite{Kor}). Aires and Kahn~\cite{AK} obtained $\de(P)\to1/2$ when $w(P)=\Om(n)$ or $|\min(P)|=\om(\log n)$. For further work, see~\cite{CPP} and~\cite{AK2}.

Our main result is a positive answer to this conjecture:
\begin{theorem}\label{cor:kahn-saks}
If $w(P)\to\oo$ then $\de(P)\to1/2$.
\end{theorem}
We derive Theorem~\ref{cor:kahn-saks} as a corollary of Theorem~\ref{thm:main} below.

The difficulty in proving $\de(P)\to1/2$ lies in finding where the balanced elements are; even a slight perturbation to a poset may change the location of a balanced pair. Previous results on balancing, especially those on the 1/3-2/3 Conjecture, have focused on pairs with similar height statistics (see Section~2), but this approach appears to work only up to $\de(P)\ge 1/e-o(1)$. We give a simpler proof that large width implies this weaker lower bound in Section 6. 

It turns out the key to finding pairs with balancing coefficient near $1/2$ is to actually look for two quite different stronger configurations, which we describe below.

For a set $Z\se P$, let $\S_Z$ be the set of linear orders on $Z$, and let $\L_{P;Z}$ be the law of a uniform linear extension restricted to $Z$.
For two laws $\mu,\nu$ on $\S_Z$, put
$$\TV(\mu,\nu)=\frac12\sum_{\si\in \S_Z}|\mu(\si)-\nu(\si)|.$$
Let $\U_k$ be the uniform law on the $k!$ orders of a $k$-element set.
We call a $k$-element set $Z\se P$ an \em{$\ep$-pseudo antichain} if
$\TV(\L_{P;Z},\U_k)\le\ep$.

Similarly, let $\I_t$ be the law assigning mass $1/(t+1)$ to each of the orders $y_1\lq\cds\lq y_j\lq x\prec y_{j+1}\lq\cds\lq y_t$ for $0\le j\le t$; equivalently $\I_t$ is the uniform extension law of the disjoint union of a $t$-element chain with an isolated element. We call a $(t+1)$-element set
$Z\se P$ an \em{$\ep$-pseudo insertor} if it has a labelling
$Z=\{x,y_1,\ldots,y_t\}$ for which
$\TV(\L_{P;Z},\I_t)\le\ep$, identifying the labels of $\I_t$ with those
of $Z$.

The following is our main result, which has a Ramsey-theoretic flavor:
\begin{theorem}\label{thm:main}
For every fixed $k\ge2$, $t\ge1$, and $\ep>0$, there is a
$W_{k,t,\ep}$ such that every poset $P$ with $w(P)>W_{k,t,\ep}$
contains at least one of the following:
\bol
\it an $\ep$-pseudo antichain of size $k$;
\it an $\ep$-pseudo insertor of size $t+1$.
\eol
\end{theorem}
To deduce Theorem~\ref{cor:kahn-saks}, apply Theorem~\ref{thm:main} with $k=2$ and $t=1$. Since $\I_1=\U_2$, either alternative gives distinct $x,y\in P$ with
$|\PP(x\prec y)-1/2|\le\ep$, and hence $\de(P)\ge1/2-\ep$. Note that both possibilities in Theorem~\ref{thm:main} are necessary: an $n$-element antichain has no pseudo-insertor with $t\ge 2$, while Theorem~\ref{thm:counterexample} shows that a poset may have arbitrarily large width but no pseudo-antichain with $k\ge 3$.

Let $\pi(x)$ be the number of $y\in P$ which are incomparable to $x$, and let $\pi(P)=\max_{x\in P}\pi(x)$. While we shall not focus on $\pi(P)$, we note that \cite[Theorem 1.7]{AK2} shows that Theorem~\ref{cor:kahn-saks} may be strengthened as follows:
\begin{theorem} If $\pi(P)\to \oo$ then $\de(P)\to 1/2$.
\end{theorem}
Theorem~\ref{thm:main} is deduced from Theorem~\ref{thm:small-gaps} in
Section~\ref{sec:width-compactness}; the latter is proved in
Section~\ref{sec:small-gap-proof}.

To state our remaining results, we first must introduce some notation. Let $F$ be uniform in the \em{order polytope} $\O(P)=\{z\in[0,1]^P:z_x\le z_y\text{ if }x<_Py\}$, and put
$$Q_x=\max_{y<x}F(y),\wsp R_x=\min_{y>x}F(y),\wsp c_x=\EE(F(x)-Q_x)=\EE(R_x-F(x)),$$
with $Q_x=0$ and $R_x=1$ when $x$ is minimal or maximal, respectively. Following a volume-preserving bijection of Stanley~\cite{Stan}, we may equivalently view $c$ as the centroid of the chain polytope $\C(P)=\{z\in [0,1]^P:\sum_{x\in C} z_x\le 1\text{ for all chains }C\se P\}$. Write $S_X=\sum_{x\in X}c_x$.

For $X\se P$ with $|X|\ge k$, define
$$\de_k(X;P)=\max_{\substack{x_1,\dots,x_k\in X\\\text{pairwise distinct}}} \min_{\pi\in S_k} \PP\ll(x_{\pi(1)}\lq\cds\lq x_{\pi(k)}\rr).$$
Write $\de_k(P)=\de_k(P;P)$, so that $\de(P)=\de_2(P)$. Note that if $Z$ is a $k$-element set with $\de_k(Z;P)\ge 1/k!-\be$, then $\TV(\L_{P;Z},\U_k)\le k!\be$.

It follows from \cite{AK} that
$$S_X=\Om(n)\imrsp \de_k(X;P)\ge 1/k!-\ep.$$
This is strengthened in the following:
\begin{theorem}\label{thm:tuples}
For every fixed $k\ge2$ and $X\se P$ with $|X|\ge k$,
$$S_X=\om(|X|^{1/2})\imrsp\de_k(X;P)\to1/k!.$$
Moreover, for every $\ep>0$, there is $\et=\et(k,\ep)>0$
such that
\begin{equation}\label{eq:fixed-accuracy}
S_X= \om(|X|^{1/2-\et})\imrsp\de_k(X;P)\ge 1/k!-\ep.
\end{equation}
\end{theorem}

Hence, \eqref{eq:antichain} and Theorem~\ref{thm:tuples} give
\begin{corollary}\label{cor:antichains}
For every $k\ge2$ and $\ep>0$, there is $\et'=\et'(k,\ep)>0$ such that, for every antichain $A\se P$,
$$|A|=\om(n^{2/3-\et'})\imrsp S_A\ge\f{|A|^2}{2n}=\om(|A|^{1/2-\et})\imrsp \de_k(A;P)\ge 1/k!-\ep.$$
In particular, $w(P)=\om(n^{2/3})$ implies $\de_k(P)\to1/k!$.
\end{corollary}

The next two constructions demonstrate the tightness of Theorem~\ref{thm:tuples}.

\begin{theorem}\label{thm:counterexample}
For every fixed $0<\ep<1/2$, there exist $\et_\ep>0$ and a sequence of posets $P$, with $n=|P|\to\oo$, such that
$$S_P=\om(n^{1/2-\ep}),\wsp\de_3(P)\le1/6-\et_\ep.$$
Furthermore, $w(P)=\om(n^{1/2-\ep})$ and $\de(P)=1/2.$
\end{theorem}

\begin{theorem}\label{thm:local-width-counterexample}
For every fixed $0<\ep<1/6$, there exist $\et_\ep>0$ and a sequence of posets $P$ with maximum antichains $A$, where $n=|P|\to\oo$, such that
$$|A|=\om(n^{2/3-\ep}),\wsp\de_3(A;P)\le1/6-\et_\ep.$$
Furthermore, $S_P=\om(n^{2/3-\ep})$ and
$\de_k(P)=1/k!$ for fixed $k$ and sufficiently large $n$.
\end{theorem}

Theorem~\ref{thm:tuples} is proved in Section~\ref{sec:fixed-accuracy}; Theorems~\ref{thm:counterexample} and~\ref{thm:local-width-counterexample} are proved in Section~\ref{sec:counterexamples}.

The following question remains open:
\begin{question}[Width and tuple balance]\label{question:width-tuples}
Does there exist a fixed $\et>0$ such that
$$w(P)=\om(n^{2/3-\et})\imrsp \de_3(P)\to1/6\ ?$$
\end{question}

\section{Common Tools}\label{sec:common-tools}
We shall work with the following statistics of $F$:
$$\mu_x=\EE F(x),\wsp \si_x^2=\Var F(x).$$
Let $f:P\to[n]$ be a uniform order-preserving bijection, so that we have the coupling $F(x)=U_{(f(x))}$ where $U_{(j)}$ are the order statistics of $n$ independent uniforms, independent of $f$. It is convenient to work with a scaled multiple of $F$ rather than $F$ itself: let
$$Z_x=(n+1)F(x).$$
Then we have the \em{height} statistic
$$h_P(x)=\EE Z_x=\EE f(x)=(n+1)\mu_x$$
and, for distinct $x,y$, 
$$\PP(x\lq y)=\PP(F(x)<F(y))=\PP(Z_x<Z_y)=\PP(f(x)<f(y)).$$

Let
$$\De(x,y):=|h(x)-h(y)|, \wsp d(x,y)=(n+1)\s{\Var(F(x)-F(y))}=\s{\Var(Z_x-Z_y)}.$$

We write $h_P, \De_P(x,y),$ etc. when the underlying poset is unclear.

Conditional on the other coordinates, $F(x)$ is uniform on $[Q_x,R_x]$. Consequently
\begin{equation}\label{eq:windows}
 \si_x^2\ge c_x^2/3,\qquad
 \EE|F(x)-F(y)|\ge c_x/2\quad(x\ne y).
\end{equation}
For a random variable $V$ whose density $g_V$ is log-concave, we have the following standard bounds (see \cite{LV} for the first two bounds and \cite[Lemma 4.4]{AK} for the third):
\begin{equation}\label{eq:inputs}
\|g_V\|_\oo\le\f{\s2}{\s{\Var V}},\wsp \PP(V\ge \EE V),\PP(V\le \EE V) \ge 1/e, \wsp\PP(|V|\le h)\le\f{3h}{\EE|V|}.
\end{equation}
Since the law of $F(x)-F(y)$ is log-concave, we have $\PP(|F(x)-F(y)|\le h)\le6h/c_x$ in particular.

If $V$ has a log-concave density, $\EE V=0$, and
$\sigma=\|V\|_2>0$, then Lov\'asz--Vempala~\cite[Lemma 5.7]{LV} gives
\begin{equation}\label{eq:logconcave-tails}
 \PP(|V|>t\sigma)\le e^{1-t}\qquad(t\ge0).
\end{equation}
The case $0\le t\le1$ is trivial. Integrating gives
$\EE V^4\le24e\,\sigma^4$.

For $x\ne y$, conditional on every coordinate except $F(x)$, the distribution of $Z_x-Z_y$ is uniform on $[(n+1)Q_x-Z_y,(n+1)R_x-Z_y]$. Total variance and Jensen give
$d(x,y)^2\ge(n+1)^2\EE(R_x-Q_x)^2/12\ge(n+1)^2c_x^2/3$.
Interchange $x,y$ to get
\begin{equation}\label{eq:fiber}
 d(x,y)\ge(n+1)\f{\max(c_x,c_y)}{\s3}.
\end{equation}

For an antichain $A$, we have the following inequality from \cite[Theorem 5.4]{AK}, which follows from observing that $c_x\ge\f{e(P)}{2ne(P-x)}$ and applying Cauchy-Schwarz with the inequality $e(P)\ge \sum_{x\in A}e(P-x)$ from \cite{EHS} or \cite{Stac} (or an equivalent geometric version from \cite{BB}):
\begin{equation}\label{eq:antichain}
\sum_{x\in A}c_x\ge\f{|A|^2}{2n}.
\end{equation}
For an interval $I$, let $T_I$ denote clipping to $I$ (so $T_{[a,b]}(x)=\min(b,\max(a,x))$). Then
\begin{equation}\label{eq:budget}\Cov(T_I(F(x)),T_I(F(y)))\ge 0,\wsp \sum_y\Cov(T_I(F(x)),T_I(F(y)))\le |I|/4.
\end{equation}
To see \eqref{eq:budget}, put $H_a(t)=\min(t,a)-at$. Conditioning on all but one uniform gives
$\Cov(\mathbf1_{\{U_{(j)}>t\}},\sum_i\mathbf1_{\{U_i>a\}})=p_j(t)H_a(t)$, where $p_j$
is the density of $U_{(j)}$. Integrating in $t$ and averaging over
$f(x)$ yields
$$
 \sum_y\Cov(F(x),\mathbf1_{\{F(y)>a\}})=\EE H_a(F(x))\le1/4.
$$
Now integrate over $a\in I\cap[0,1]$. FKG applies to the uniform measure
on the order polytope; since both $T_I(t)$ and $t-T_I(t)$ are
increasing, it gives the claimed bound and nonnegativity of every summand.

The following is a direct consequence of the continuous form of Shepp's XYZ inequality \cite{She82} (see also \cite{CP23} for a version for weak order-preserving maps): for any distinct $x,y,z$,
\begin{equation}\label{eq:cov}
\Cov(F(y)-F(x),F(z)-F(y))\le0.
\end{equation}
In particular, 
\begin{equation}\label{eq:sub-add}
d(x,y)^2\le d(x,z)^2+d(z,y)^2.
\end{equation}

For any event $E$ with $\PP(E)>0$ and any variable $V$ of finite variance,
Cauchy--Schwarz gives
\begin{equation}\label{eq:weak-conditioning}
(\EE[V\mid E]-\EE V)^2\le\Var(V)(1/\PP(E)-1),\wsp \Var(V\mid E)\le\Var(V)/\PP(E).
\end{equation}

Finally, we need a recent result of Haqi \cite{Haqi}, solving affirmatively a conjecture of Kahn (see for instance \cite{BT}): for a nonempty ideal $I$ of $P$,
\begin{equation}\label{eq:Haqi}
\max_{x\in I} h(x)\ge |I|-|\max(I)|+1
\end{equation}
Here $\min(I)$ and $\max(I)$ are the sets of minimal and maximal elements, respectively. Applying \eqref{eq:Haqi} to both $P$ and its dual, we derive the following easy corollary: for any nonempty proper ideal $I$ of $P$,
\begin{equation}\label{eq:local-gap-input}
 \min_{y\in I^c}h(y)-\max_{x\in I}h(x)\le|\max(I)|+|\min(I^c)|-1.
\end{equation}

\section{Koml\'os selection}
The fixed-$k$ Koml\'os theorem \cite[Theorem 8.1]{AK}, extending
\cite{Komlos}, yields the following form.
\begin{theorem}[Koml\'os selection]\label{thm:selection}
Fix $k\ge2$ and $A,B,\ep>0$. There is an integer $M_{k,A,B,\ep}$ such that the following holds. Let $V_1,\dots,V_m$ be jointly distributed real random variables, with $m\ge M_{k,A,B,\ep}$, satisfying
\begin{equation}\label{eq:selection}
 \EE V_i^2\le A,\qquad
 \PP(|V_i-V_j|\le h)\le Bh\quad(i\ne j,\ h>0).
\end{equation}
Then there are distinct
$i_1,\cds,i_k$ such that, for every $\pi\in S_k$,
$$\PP(V_{i_{\pi(1)}}<\cds <V_{i_{\pi(k)}})\ge 1/k!-\ep.
$$
\end{theorem}
\begin{proof}
Fix $k\ge 2$ and $A,B,\ep>0$. Choose $T$ large and $h>0$ small enough so that $kA/T^2+\b{k}{2}Bh\le \ep/2$. Partition $[-T,T]$ into intervals of length at most $h$ and add the two tail intervals. Then, for $V_{i_1},\dots, V_{i_k}$, the probability that two $V_i$'s are in the same interval is at most
$$\sum_a \PP(|V_{i_a}|>T)+\sum_{a<b} \PP(|V_{i_a}-V_{i_b}|\le h)\le kA/T^2+\b{k}{2}Bh\le \ep/2.$$

Let $B_i$ be the index of the interval containing $V_i$, with the
intervals indexed from left to right. Apply \cite[Theorem 8.1]{AK}
to the finite-valued variables $B_i$, with error $\ep/4$. For sufficiently
large $m$, it supplies distinct indices $i_1,\dots,i_k$ for which the
probabilities $p_\pi=\PP(B_{i_{\pi(1)}}<\cdots<B_{i_{\pi(k)}})$ differ from one another by at most $\ep/2$. The collision estimate gives $\sum_\pi p_\pi\ge1-\ep/2$, so $p_\pi\ge\f{1-\ep/2}{k!}-\ep/2\ge1/k!-\ep$. A strict order of interval indices implies the same strict order of the
original variables. Hence $\PP(V_{i_{\pi(1)}}<\cdots<V_{i_{\pi(k)}})\ge1/k!-\ep$
for every $\pi$, as required.
\end{proof}

The following consequence of Theorem~\ref{thm:selection} will be used in
Sections~\ref{sec:width-compactness} and~\ref{sec:small-gap-proof}. Put 
$$\ell(x,y)=\De(x,y)+d(x,y).$$

\begin{lemma}\label{cp:lem:doubling}
If $P$ has no $\ep$-pseudo antichain (in particular if $\de_k(P)<1/k!-\ep$), every $\ell$-ball of radius $r$ is covered by at most $N_{k,\ep}$ balls of radius $r/2$.
\end{lemma}
\begin{proof}
Take an $r/2$-separated subset of the $\ell$-ball centered at $b$. Partition the heights into a bounded number of intervals of length $r/8$. Within one interval, distinct points satisfy
$$d(s,t)=\ell(s,t)-\Delta(s,t)>r/2-r/8=3r/8.$$

Apply Theorem~\ref{thm:selection} with the fixed $k$ to
$X_z=(Z_z-Z_b)/r$ for the vertices in this interval. Since they lie in the
ball,
$$\EE X_z^2=\f{\Delta(z,b)^2+d(z,b)^2}{r^2}\le1.$$

Moreover, $\sd(X_s-X_t)>3/8$, so log-concavity and \eqref{eq:inputs} give
$$\PP(|X_s-X_t|\le u)\le Cu\wsp(u>0).$$

These are precisely the two hypotheses of Theorem~\ref{thm:selection}. If the interval contained sufficiently many selected vertices, it would give a $k$-tuple with every order probability at least $1/k!-\ep/(2k!)$, and hence an $\ep$-pseudo antichain, contrary to the hypothesis. Thus the number of selected vertices in each interval, and hence in the ball, is bounded in terms of $k,\ep$. A maximal $r/2$-separated subset supplies centers for the required covering.
\end{proof}

In particular, balls of bounded radius have covering numbers at most
$C r^{-D}$ at scale $0<r\le1$, for constants depending only on
$k,\ep$.

\section{Proof of Theorem~\ref{thm:tuples}}
\label{sec:fixed-accuracy}
\begin{proof}[Proof of Theorem~\ref{thm:tuples}]
Fix $k\ge2$, and put
\[
 K_X:=\sup_{\substack{0<u\le1/2\\b\in\RR}}
 u\,|\{x\in X:\mu_x\in[b,b+u),\ u\le\si_x<2u\}|.
\]
We first show that, for every nonempty $X\se P$ and
$0<s\le1/2$,
\begin{equation}\label{eq:count}
 |\{x\in X:\si_x\ge s\}|\le CK_Xs^{-2},
\end{equation}
where $C$ is an absolute constant.
At each dyadic scale $u=2^js\le1/2$, partition $[0,1]$ into
$O(1/u)$ intervals of length $u$. Each contains at most
$K_X/u$ elements with $u\le\si_x<2u$. Summing
$O(K_X/u^2)$ proves \eqref{eq:count}. Since $c_x\le\s3\si_x$
and $c_x\le1/2$, integration gives
$$S_X=\int_0^{1/2}|\{x\in X:c_x>t\}|\,dt\le\int_0^\oo\min\{|X|,CK_Xt^{-2}\}\,dt
 =O(\s{|X|K_X}).
$$
Hence $K_X=\Om(S_X^2/|X|)$, so $S_X=\om(|X|^{1/2})$ implies $K_X\to \oo$.

If $K_X\to\oo$, we now show that $\de_k(X;P)\to1/k!$. We will find a growing subset of $X$
whose normalized coordinates satisfy the two hypotheses of
Theorem~\ref{thm:selection}: bounded second moments and a uniform
small-ball bound for every pair.

By the definition of $K_X$, choose $u,b$ so that
$$A=\{x\in X:\mu_x\in[b,b+u),\ u\le\si_x<2u\},\wsp m=|A|,\wsp mu\ge K_X/2\longrightarrow\oo.$$
For $x\in A$, put $V_x=(F(x)-b)/u$. Then
$0\le\EE V_x<1$ and $1\le\Var V_x<4$, so
$\EE V_x^2\le5$. It remains to find many of these coordinates
whose pairwise differences do not concentrate near zero.

Put $Y_x=T_{[-10,10]}(V_x)$. The density bound in
\eqref{eq:inputs} gives $\|g_{V_x}\|_\oo\le\s2$, and
Markov's inequality gives $\PP(|V_x|>10)\le1/20$.
For any interval $J$ of length $1/4$, therefore,
$$\PP(Y_x\in J)\le\PP(V_x\in J)+\PP(|V_x|>10)\le\f{\s2}{4}+\f1{20}<\f12.$$
Taking $J=[\EE Y_x-1/8,\EE Y_x+1/8]$ shows that
$\Var Y_x\ge(1/8)^2/2=:v=1/128$.

For $I=[b-10u,b+10u]$, we have
$Y_x=(T_I(F(x))-b)/u$. Since every covariance in
\eqref{eq:budget} is nonnegative, restricting its sum to $A$ gives
$$
 \sum_{y\in A}\Cov(Y_x,Y_y)
 =u^{-2}\sum_{y\in A}\Cov(T_I(F(x)),T_I(F(y)))
 \le u^{-2}\f{20u}{4}=\f5u.
$$
Join distinct $x,y\in A$ by an edge when
$\Cov(Y_x,Y_y)>v/2$. Each vertex has degree at most
$10/(vu)$. Greedily keeping a vertex and deleting its neighbors
therefore produces a set $B\se A$ with
$$
 |B|\ge\f{m}{1+10/(vu)}=\Om(mu)\longrightarrow\oo,
 \qquad \Cov(Y_x,Y_y)\le v/2\quad(x\ne y\in B).
$$
Here the implied constant is absolute, since $u\le1/2$. For distinct $x,y\in B$,
$$\EE(Y_x-Y_y)^2\ge\Var Y_x+\Var Y_y-2\Cov(Y_x,Y_y)\ge v.$$
Clipping is $1$-Lipschitz and $|Y_x-Y_y|\le20$, hence
$$
 \EE|V_x-V_y|\ge\EE|Y_x-Y_y|
 \ge\f{\EE(Y_x-Y_y)^2}{20}\ge\f v{20}.
$$
The difference $V_x-V_y=(F(x)-F(y))/u$ has a log-concave
density. Thus \eqref{eq:inputs} gives, for every $h>0$,
$$
 \PP(|V_x-V_y|\le h)
 \le\f{3h}{\EE|V_x-V_y|}\le\f{60}{v}h.
$$
Together with $\EE V_x^2\le5$ and $|B|\to\oo$, this verifies
all hypotheses of Theorem~\ref{thm:selection}. It supplies $k$
elements of $B\se X$ with every ordering probability
$1/k!+o(1)$. The common positive rescaling preserves their orders,
and the orders under $F$ have the same law as under $f$.
Therefore $\de_k(X;P)\to1/k!$. Together with the bound on $K_X$, this proves the first assertion.

Now fix $\ep>0$, and suppose $\de_k(X;P)<1/k!-\ep$; all constants below may depend on $k,\ep$. The preceding argument gives $K_X=O_{k,\ep}(1)$.
Thus \eqref{eq:count} implies
\begin{equation}\label{eq:fixed-variance-count}
 |\{x\in X:\si_x>s\}|\le C s^{-2}\qquad(s>0).
\end{equation}
We refine the count by covering in the difference metric $d$ in rank units.

Choose an integer $L\ge2$ large enough that the finite selection
theorem applies to $L$ variables satisfying
$\EE V_i^2\le5$ and
$\PP(|V_i-V_j|\le h)\le4\s2\,h$, with target error
$\ep/2$. Put $D=\log_2 L$. We claim that, if $0<u\le s$
and $Y\se X$ satisfies
$$
 \mu_x\in[b,b+u),\qquad \si_x\le s,\qquad c_x\ge u
 \quad(x\in Y),
$$
then
\begin{equation}\label{eq:fixed-packing}
 |Y|\le C(s/u)^D.
\end{equation}
Consider a ball in $Y$, centered at $a\in Y$, of radius
$r\ge(n+1)u/2$. For any subset in this ball with pairwise distances
greater than $r/2$, set $V_x=(Z_x-Z_a)/r$. Then
$$
 \EE V_x^2=\f{d(x,a)^2+(n+1)^2(\mu_x-\mu_a)^2}{r^2}\le5,
 \qquad
 \PP(|V_x-V_y|\le h)
 \le\f{2\s2\,hr}{d(x,y)}\le4\s2\,h.
$$
The second bound follows from the log-concave density bound in
\eqref{eq:inputs}. Subtracting the same random coordinate preserves
all orders. The deficit therefore forces this separated subset to
have fewer than $L$ elements. Taking a maximal separated subset
covers the ball by at most $L$ balls of radius $r/2$, with centers
in $Y$.

The diameter of $Y$ is at most $2(n+1)s$. Start with a ball of radius
$2^J(n+1)u$, where $J=\lceil\log_2(2s/u)\rceil$, and apply the
covering step $J+2$ times. The final radii are $(n+1)u/4$.
By \eqref{eq:fiber}, distinct points of $Y$ have distance at least
$(n+1)u/\s3>(n+1)u/2$, so each final ball contains at most one point.
Hence $|Y|\le L^{J+2}\le C(s/u)^D$, proving
\eqref{eq:fixed-packing}.

For $0<u\le1/2$, separate the elements with $c_x\ge u$ into
those with $\si_x>s$ and those with $\si_x\le s$.
Use \eqref{eq:fixed-variance-count} for the first group and partition
the second group's coordinate means into $O(1/u)$ intervals of length $u$.
Equation~\eqref{eq:fixed-packing} gives, for any $s\ge u$,
$$
 |\{x\in X:c_x\ge u\}|
 \le C\ll(s^{-2}+s^D u^{-(D+1)}\rr).
$$
Choose $s=u^{(D+1)/(D+2)}$. Both terms are $u^{-p}$, where
$p=2(D+1)/(D+2)\in(1,2)$. Integrating this count, using
$c_x\le1/2$, yields
$$S_X\le\int_0^\oo\min\{|X|,Cu^{-p}\}\,du\le C' |X|^{1-1/p}=C'|X|^{1/2-\ga},\wsp \ga=\f1{2(D+1)}>0.$$
Taking $\et=\ga$, the contrapositive proves~\eqref{eq:fixed-accuracy}.
\end{proof}

\section{Counterexamples for window and local width thresholds} \label{sec:counterexamples}

Let $C_t$ be the chain of size $t$. We use $P+Q$ to denote the parallel sum (disjoint union) of $P$ and $Q$, and $P \op Q$ for the series sum (i.e. every element of $P$ precedes every element of $Q$. Both constructions in this section are series-parallel, meaning they can be formed from chains by $+$ and $\op$.

Note that both constructions generalize to $\de_k$, as for every fixed $k\ge3$, restricting a $k$-tuple to three elements gives, for any set $X$ of size at least $k$,
$$\de_k(X;P)\le\f6{k!}\de_3(X;P)$$

\subsection{A recursive example without balanced triples}
Fix $0<\th<1$ and a sufficiently large integer $N$. Starting from
$P_0=C_N$, put
\begin{equation}\label{eq:counter-construction}
 n_r=|P_r|,\qquad p_r=\lceil\th n_r\rceil,\qquad
 B_r=C_{p_r}\op P_r\op P_r\op C_{p_r},\qquad
 P_{r+1}=B_r+B_r.
\end{equation}
Copies in these expressions are disjoint. Clearly $\de(P_r)=1/2$ when $r\ge 1$.

Taking an $n$-element component in a series sum of total size $b$ multiplies each $c_x$ by $(n+1)/(b+1)$, while parallel composition leaves $c_x$ unchanged. Each end-chain element contributes $1/(2n_r+2p_r+1)$. Thus

$$n_{r+1}=4(n_r+p_r),\wsp w(P_r)=2^r,\wsp S_{P_{r+1}}=\f{4(n_r+1)}{2n_r+2p_r+1}S_{P_r}+\f{4p_r}{2n_r+2p_r+1}.$$

The size multiplier tends to $4(1+\th)$ and the $S_P$ multiplier to $2/(1+\th)>1$, with summable errors $O(1/n_r)$. Consequently, for fixed $N$,
\begin{equation}\label{eq:counter-result}
 n_r=\Th\bigl((4(1+\th))^r\bigr),\qquad
 w(P_r)=\Th(n_r^{\be_\th}),\qquad
 S_{P_r}=\Th(n_r^{a_\th}),
\end{equation}
where
$$\be_\th=\f{\log2}{\log(4(1+\th))},\wsp a_\th=3\be_\th-1=\f{\log(2/(1+\th))}{\log(4(1+\th))}.$$
Both exponents approach $1/2$ from below as $\th\to0$.

Write $U\le_{\rm st}V$ if $\PP(U>t)\le\PP(V>t)$ for every $t$.
We use the following observation.
\begin{lemma}[Nearly uniform triples]\label{lem:counter-independence}
Let $X,Y,Z$ have continuous, stochastically ordered marginal laws,
with $Z$ independent of $(X,Y)$. Suppose the conditional law of either
of $X,Y$, given the other, increases stochastically with the conditioning
value. If each of the six orders has probability at least $1/6-\ep$,
where $\ep\ge0$, then, for every $s,t\in\RR$,
\begin{equation}\label{eq:counter-rectangles}
 0\le\PP(X\le s,Y\le t)-\PP(X\le s)\PP(Y\le t)\le48\ep.
\end{equation}
\end{lemma}
\begin{proof}
Write $H_X,H_Y,H_Z$ for the distribution functions, and put
$U=H_X(X)$, $V=H_Y(Y)$, $W=H_Z(Z)$. These variables are uniform
on $[0,1]$, with $W$ independent of $(U,V)$. Stochastic ordering
makes $\{X<Z\}$ and $\{U<W\}$ nested almost surely, so
$$
 \PP\bigl(\{X<Z\}\mathbin{\triangle}\{U<W\}\bigr)
 =|\PP(X<Z)-1/2|\le3\ep.
$$
The same holds for $Y,Z$. Thus
$q:=\PP(\min(U,V)<W<\max(U,V))\ge1/3-8\ep$.
Put $D(u,v)=\PP(U\le u,V\le v)-uv$. Conditional monotonicity
makes $D$ concave in each coordinate, with zero boundary values,
so $D\ge0$. Independence of $W$ gives
$$
 \int_0^1D(t,t)\,dt=\f{1/3-q}{2}\le4\ep.
$$
Concavity in the two coordinates also gives
$D(t,t)\ge\min\{t,1-t\}^2D(u,v)$ for all $u,v,t\in[0,1]$.
Integrating yields $D(u,v)\le48\ep$; take $u=H_X(s)$ and
$v=H_Y(t)$ to obtain~\eqref{eq:counter-rectangles}.
\end{proof}

The order-polytope density is log-supermodular, since its support is closed
under coordinatewise minimum and maximum. The Ahlswede--Daykin
Four-Functions theorem~\cite{AD}, applied on grids and then by approximation,
gives the same property for each two-coordinate marginal density $g$:
\[
 g(x,s)g(x',t)\ge g(x,t)g(x',s)
 \qquad(x<x',\ s<t).
\]
Integrating over $s\le a<t$ and normalizing shows that
$\PP(Y>a\mid X=x)$ is nondecreasing in $x$. Exchanging $X,Y$
gives conditional stochastic monotonicity in both directions.

\begin{proof}[Proof of triple bound in Theorem~\ref{thm:counterexample}]
The coordinate laws of $P_r$ are totally stochastically ordered:
this holds for a chain, series sums preserve it by their common
increasing affine embeddings, and parallel duplication repeats the same
list of laws. Put $n=n_r$, $p=p_r$, $b=2n+2p$. An old copy preceded
by $a\in\{p,p+n\}$ has embedding
\begin{equation}\label{eq:counter-dirichlet}
 F'_i=L+TF_i,\qquad
 (L,T,1-L-T)\sim\operatorname{Dirichlet}(a,n+1,b-a-n),
\end{equation}
independently of its old coordinates. Total variance and Dirichlet
moments give
$$
 2b\Var(F'_i)\le
 \f{2b(n+1)(n+2)}{n(b+1)(b+2)}\,n\Var(F_i)+1/2.
$$
The coefficient tends to $1/(1+\th)<1$, and end-chain coordinates
have scaled variance at most $1/2$. For sufficiently large $N$,
iteration therefore gives
\begin{equation}\label{eq:counter-variance}
 \sup_{r,x\in P_r}n_r\Var(F_x)\le C_\th.
\end{equation}

Consider an antichain triple with least ordering probability $1/6-\xi$.
Restrict it to its smallest recursive component $P_s$, preserving its
ordering law, and put $n=n_{s-1}$, $p=p_{s-1}$, $b=2n+2p$,
so $m=|P_s|=2b$. Two elements $x,y$ share an old copy in one branch,
while $z$ lies in the other and is independent of them.
In~\eqref{eq:counter-dirichlet}, put $R=L+T$. The old covariance is nonnegative by FKG, and $\Var L,\Var R\ge\Cov(L,R)$.
Since conditional means are convex combinations of $L,R$, total covariance gives
\begin{equation}\label{eq:counter-shared-covariance}
 m\Cov(F_x,F_y)\ge 2b\Cov(L,R)
 =\f{2b\,p(n+p)}{(b+1)^2(b+2)}\ge c_\th>0.
\end{equation}

Put $A=\s m(F_x-\mu_x)$ and $B=\s m(F_y-\mu_y)$. Then
$\Var A,\Var B\le C_\th$ and $\Cov(A,B)\ge c_\th$.
The tail bound~\eqref{eq:logconcave-tails} and Cauchy--Schwarz allow a fixed
$K=K_\th\ge\s{C_\th}$ such that clipping both variables to $[-K,K]$
leaves covariance at least $c_\th/2$.
Apply Lemma~\ref{lem:counter-independence} to $(F_x,F_y,F_z)$.
Its rectangle bound is preserved by separate increasing affine changes
of coordinates. Integrating the indicator covariances over $[-K,K]^2$
therefore gives
$$
 c_\th/2\le\Cov\bigl(T_{[-K,K]}(A),T_{[-K,K]}(B)\bigr)
 \le48\xi(2K)^2.
$$
Hence $\xi\ge\et_\th:=c_\th/(384K_\th^2)>0$, uniformly in the
depth and in sufficiently large seed sizes. Since $\et_\th\le1/384$,
triples containing a comparable pair satisfy the same bound trivially. Finally, choose $\th$ so that $a_\th>1/2-\ep$, and fix such an $N$.
Equation~\eqref{eq:counter-result}, together with $\de(P_r)=1/2$
for $r\ge1$, proves Theorem~\ref{thm:counterexample}.
\end{proof}

\subsection{A maximum antichain without balanced triples}
\begin{proof}[Proof of Theorem~\ref{thm:local-width-counterexample}]
Take $H=P_r$ from the preceding construction, with $M=|H|$ and
$w=2^r=\Th(M^{\be_\th})$. Choose one seed-chain vertex and,
at each step, its descendants in the first old copy of each branch.
They form a maximum antichain $D$ of size $w$. Symmetry gives
identical marginals on $D$, and \eqref{eq:counter-variance} gives
$$\EE F_x=\mu,\qquad \Var(F_x)\le C_\th/M\quad(x\in D).$$

Fix a sufficiently large constant $L=L(\th)$, and put
$$h=\lc L\s M\rc,\wsp T=\lf M/(2h)\rf.$$
Put
\begin{equation}\label{eq:local-counter-construction}
 Q_j=C_{jh}\op H\op C_{M-jh}\hsp(1\le j\le T),\wsp P=Q_1+\cds+Q_T,\wsp A=\bcup_{j=1}^T D^{(j)}.
\end{equation}
Since the first vertices of the prefix chains are independent and identically
distributed, any $k$ of them have all $k!$ orders equally likely, so $\de_k(P)=1/k!$ whenever $T\ge k$.

Every branch has size $2M$ and width $w$, so $A$ is a maximum antichain. Its groups have coordinate mean spacing $h/(2M+1)$, while all their
variances are $O_\th(1/M)$:
\begin{equation}\label{eq:local-counter-marginals}
 \mu_j=\f{jh+(M+1)\mu}{2M+1},\qquad
 \Var(F_x)\le C_1/M\quad(x\in D^{(j)}).
\end{equation}
Indeed the embedding is $F'_x=U+VF_x$, with
$(U,V,1-U-V)\sim\operatorname{Dirichlet}(jh,M+1,M-jh)$;
its added variance is at most $1/[4(2M+2)]$.
For $i<j$, independence across branches and Chebyshev give
$$
 \PP(F_y<F_x)\le18C_1/L^2\le1/4
 \qquad(x\in D^{(i)},\ y\in D^{(j)}),
$$
once $L$ is large enough. Any triple meeting different branches
has three orders containing such a reversed pair, so one has
probability at most $1/12$. A triple within a branch retains its
ordering law in $H$. Therefore
\begin{equation}\label{eq:local-counter-deficit}
 \de_3(A;P)\le1/6-\et',\qquad
 \et'=\min\{\et_\th,1/12\}>0.
\end{equation}

Since $T=\Th(\s M)$, we have
$$n=2MT=\Th(M^{3/2}),\wsp |A|=Tw=\Th\ll(n^{(1+2\be_\th)/3}\rr).$$

Each branch has $M$ end-chain elements, each contributing
$1/(2M+1)$, while the old $c_x$ values are multiplied by
$(M+1)/(2M+1)$. Thus
\begin{equation}\label{eq:local-total-window}
 S_P=T\f{(M+1)S_H+M}{2M+1}
     =\Th\bigl(n^{2\be_\th-1/3}\bigr).
\end{equation}
Choose $\th$ so that $\be_\th>1/2-\ep/2$. Then both
$|A|$ and $S_P$ are $\om(n^{2/3-\ep})$, as required.
\end{proof}

\section{Reduction to small gaps}\label{sec:width-compactness}

For every nonempty $B\se P$, define its height span and its diameter in $d$ by
\[
 H(B):=\max_{x,y\in B}\Delta(x,y),\qquad
 D(B):=\max_{x,y\in B}d(x,y).
\]
We write $H_P(B)$ and $D_P(B)$ when the underlying poset is unclear.

\begin{lemma}\label{cp:lem:antichain-dispersion}
Fix $k\ge2$ and $\ep>0$. There are constants $C>0$ and $\et>0$ such that, if $\de_k(P)<1/k!-\ep$, every nonempty antichain $B\se P$ satisfies
\begin{equation}\label{cp:eq:antichain}
 |B|\le C\ll(1+H(B)+D(B)\rr)^{1-\et}.
\end{equation}
\end{lemma}
\begin{proof}
Put $m=|B|$, $T=H(B)+D(B)$, and
$$A=\{x\in B:(n+1)c_x\ge m/4\}.$$
We have $|A|\ge m/2$: otherwise $q=|B\x A|>m/2$, and \eqref{eq:antichain} gives
$$q^2/2\le\sum_{x\in B\setminus A}(n+1)c_x<qm/4,$$
a contradiction. By~\eqref{eq:fiber}, distinct $x,y\in A$ satisfy
$\ell(x,y)\ge d(x,y)\ge m/(4\sqrt3)$, while the $\ell$-diameter of $A$
is at most $T$. The doubling bound in Lemma~\ref{cp:lem:doubling}
therefore gives constants $C_0,p>0$, depending only on $k,\ep$, such that
\[
 m\le C_0(1+T/m)^p.
\]
If $T\le m$, this bounds $m$ by a constant. If $T>m$, it gives
$m^{p+1}\le C_0 2^p T^p$. Thus in both cases
$m\le C(1+T)^{p/(p+1)}$, proving the claim with $\eta=1/(p+1)$.
\end{proof}

If the elements of $P$ are indexed so that $h(x_1)\le h(x_2)\le \dots\le h(x_n)$, then let the \em{gap} of $P$ be
$$G(P)=\max(h(x_1),h(x_2)-h(x_1),\dots,h(x_n)-h(x_{n-1}),n+1-h(x_n)).$$
Also, put
\begin{equation}\label{cp:eq:R}
 R(P)=\max\bigl(\{1\}\cup\{d(x,y):\Delta(x,y)\le d(x,y)\}\bigr).
\end{equation}
We write $R=R(P)$ and $G=G(P)$ when the poset $P$ is clear.

\begin{lemma}\label{cp:lem:scale}
Put $M=\max(R,G)$. For all vertices $x,y$,
\begin{equation}\label{cp:eq:diffusiveM}
 d(x,y)^2\le 4M\bigl(\Delta(x,y)+M\bigr),
\end{equation}
and for every interval $I$ of length $M$, $|\{z\in P:h(z)\in I\}|\le 8M+1$.
\end{lemma}
\begin{proof}
First, the definition of $R$ gives, for any vertices $u,v$,
\[
 d(u,v)\le\max\{R,\Delta(u,v)\}.
\]
Indeed, if $d(u,v)\ge\Delta(u,v)$, this pair enters the maximum
defining $R$, so $d(u,v)\le R$. Otherwise $d(u,v)<\Delta(u,v)$.

The case $x=y$ is immediate. For distinct $x,y$, interchange them
if necessary so that $h(x)\le h(y)$, and list the vertices between
them in nondecreasing height order, breaking ties arbitrarily.
Starting at $x_0=x$, choose the next anchor to be the first vertex
whose height is at least $M$ larger than the current anchor's height.
If no such vertex occurs before or at $y$, finish with $y$ as the
last anchor. Stop as soon as $y$ has been selected, and write the
resulting anchors as $x_0,\ldots,x_q=y$. Every step except possibly the last advances the height by at least $M$. Since the advances sum to $\Delta(x,y)$, we have
\[
 (q-1)M\le\Delta(x,y),\qquad
 q\le1+\frac{\Delta(x,y)}M.
\]
A step chosen by the threshold rule advances the height by at most
$M+G\le2M$: just before the first vertex crossing the threshold,
the advance is less than $M$, and the next gap is at most $G$.
A final step that does not reach the threshold has advance less
than $M$. Thus every anchor step satisfies
\[
 d(x_{i-1},x_i)
 \le\max\{R,\Delta(x_{i-1},x_i)\}\le2M.
\]
Equation~\eqref{eq:sub-add} now gives
$$d(x,y)^2\le\sum_{i=1}^q d(x_{i-1},x_i)^2\le q(2M)^2\le4M^2\left(1+\f{\De(x,y)}M\right)=4M\bigl(\Delta(x,y)+M\bigr).$$

This proves~\eqref{cp:eq:diffusiveM}. For the vertex count, choose an anchor $b$ in a nonempty interval of length $M$. For each other vertex $z$ there,
$d(z,b)\le\max\{R,\Delta(z,b)\}\le M$, hence
$\EE(Z_z-Z_b)^2\le2M^2$ and $\PP(|Z_z-Z_b|\le2M)\ge1/2$.

Condition on the linear extension and write $f(b)=j$. The positive differences
from $Z_b$ have distributions $(n+1)\operatorname{Beta}(k,n+1-k)$,
$1\le k\le n-j$. Their summed density is bounded by the sum over
$1\le k\le n$, which before scaling is
\[
 \sum_{k=1}^n
 \f{n!}{(k-1)!(n-k)!}t^{k-1}(1-t)^{n-k}=n.
\]
Thus the summed density is at most $n/(n+1)$ on the positive half-line;
the negative half-line is identical. Averaging over the extension gives,
for every interval $I\subset\mathbb R$,
\[
 \sum_{z\ne b}\PP(Z_z-Z_b\in I)\le\f n{n+1}|I|.
\]
Apply this to $I=[-2M,2M]$. Each vertex in the chosen interval other
than $b$ contributes at least $1/2$, so there are at most $8M$ such vertices.
Including $b$ gives the asserted bound $8M+1$.
\end{proof}

\Needspace{12\baselineskip}
\begin{lemma}\label{cp:lem:localization}
Fix $k\ge2$ and $\ep,\theta>0$. For all sufficiently large
$M=\max\{R(P),G(P)\}$, if $P$ has an internal height gap of size at least
$\th M$, then $\de_k(P)\ge1/k!-\ep.$
\end{lemma}
\begin{proof}
Suppose instead that $\de_k(P)<1/k!-\ep$. Let $g\ge\theta M$ be the
size of an internal height gap, with endpoints $u,v$, so
$h_P(v)-h_P(u)=g$, and put
\[
 I=\{z:h_P(z)\le h_P(u)\},\qquad J=P\setminus I,
 \qquad L=AM\log^2M,
\]
where $A$ is a sufficiently large absolute constant.
Use the original coordinates $Z_z=(n+1)F(z)$ throughout.
By Lemma~\ref{cp:lem:scale} and~\eqref{eq:logconcave-tails}, applied to
$Z_y-Z_x-H$, for
$H=h_P(y)-h_P(x)\ge CM$,
\[
 \PP_P(Z_y<Z_x)
 \le\exp\left(1-\f{H}{d_P(x,y)}\right)
 \le C\exp(-c\sqrt{H/M}).
\]
Form $Q'$ by adjoining, all at once, the comparisons
\[
 z<u\quad\text{if }h_P(z)\le h_P(u)-L,
 \qquad
 v<z\quad\text{if }h_P(z)\ge h_P(v)+L.
\]
These comparisons stay within $I$ and $J$. Let $E$ be the event that all comparisons are satisfied for $F$.
Partition the values below $h_P(u)-L$ and above $h_P(v)+L$ into intervals of length $M$. Each interval contains at most $8M+1$ vertices by
Lemma~\ref{cp:lem:scale}. The union bound therefore gives
$$\PP_P(E^c)\le CM\sum_{j\ge\lfloor L/M\rfloor}e^{-c\sqrt j}\le M^{-4},$$
provided $A$ and then $M$ are sufficiently large.
Thus $E$ has positive probability, the added comparisons are consistent,
and conditioning gives the uniform order-polytope law of $Q'$. Since the
ground set and rank scale are unchanged, $h_{Q'}(z)=\EE_P[Z_z\mid E]$.
No comparison is added from $J$ to $I$, so $I$ remains an ideal.

For $s\in I,t\in J$, write $H=h_P(t)-h_P(s)\ge g\ge\theta M$.
Equation~\eqref{cp:eq:diffusiveM} gives
$d_P(s,t)/H\le C_\theta$.
Putting $q=\PP_P(E^c)\le M^{-4}$ in~\eqref{eq:weak-conditioning}, we obtain
\[
 \left|\frac{h_{Q'}(t)-h_{Q'}(s)}{h_P(t)-h_P(s)}-1\right|
 \le C_\theta\sqrt{\frac q{1-q}}\le\frac12
\]
for sufficiently large $M$ in terms of $\theta$. Hence
\[
 \min_{t\in J}h_{Q'}(t)-\max_{s\in I}h_{Q'}(s)\ge g/2.
\]

Every vertex of $I$ whose original height is at most $h_P(u)-L$ now lies
below $u$, and every vertex of $J$ whose original height is at least
$h_P(v)+L$ now lies above $v$. Thus the original heights from
$\max_{Q'}I$ lie in $(h_P(u)-L,h_P(u)]$, and those from
$\min_{Q'}J$ lie in $[h_P(v),h_P(v)+L)$. By applying~\eqref{eq:local-gap-input} in $Q'$, we have
$$g/2\le |\max_{Q'}I|+|\min_{Q'}J|-1.$$

Choosing the larger of these antichains gives $B\se P$ with
\begin{equation}\label{cp:eq:localized-antichain}
 |B|\ge g/4\ge\theta M/4,\qquad H_P(B)\le L=AM\log^2M.
\end{equation}
Equation~\eqref{cp:eq:diffusiveM} then gives $D_P(B)=O(M\log M)$.
Lemma~\ref{cp:lem:antichain-dispersion}, for a fixed $\eta=\eta(k,\ep)>0$,
instead gives
\[
 |B|=O_{\theta,k,\ep}\bigl((M\log^2M)^{1-\eta}\bigr)=o(M),
\]
contradicting~\eqref{cp:eq:localized-antichain} for sufficiently large $M$.
\end{proof}

\begin{theorem}\label{thm:small-gaps}
For every fixed $k\ge2$, $t\ge1$, and $\ep>0$, there exists
$R_*=R_*(k,t,\ep)>0$ such that every finite poset $P$ with
$G(P)<R(P)$ and $R(P)\ge R_*$ contains either an $\ep$-pseudo antichain of
size $k$ or an $\ep$-pseudo insertor of size $t+1$.
\end{theorem}

Theorem~\ref{thm:main} follows as a corollary.
\begin{proof}[Proof of Theorem~\ref{thm:main}]
Fix $k\ge2$, $t\ge1$, and $\ep>0$, and put
$\gamma=\min\{\ep/2,1/(2(t+1))\}$ and $\beta=\gamma/(2k!)$.
Suppose $P_0$ contains no $\ep$-pseudo antichain of size $k$.
Adjoin a unique minimum and maximum to obtain $P$; its width and the
relative-order laws on old vertices are unchanged. Any tuple containing
a new vertex has a deterministic pair order, and hence TV distance at
least $1/2$ from $\U_k$. Thus $P$ contains no $\gamma$-pseudo antichain
of size $k$, and $\de_k(P)<1/k!-\beta$.

Let $R_*=R_*(k,t,\gamma)$ be the threshold in Theorem~\ref{thm:small-gaps}, and let $M_0$ be the threshold in Lemma~\ref{cp:lem:localization} for $k,\beta,\theta=1$.
Put $M=\max(R,G)$. For every vertex $x$ and a neighbor $y$ in height order, $(n+1)c_x\le\sqrt3\,d(x,y)\le\sqrt3M$.
The antichain window bound gives $w(P_0)=w(P)\le2\sqrt3M$.
If
\[
 w(P_0)>2\sqrt3\max\{M_0,R_*,1\},
\]
then $M>M_0$, $M>R_*$, and $M>1$. If $G=M$, the largest height gap
is internal, since the endpoint gaps are one. Lemma~\ref{cp:lem:localization}
with $\theta=1$ would then give $\de_k(P)\ge1/k!-\be$, a contradiction.
Thus $G<M$, so $M=R$, $G<R$, and $R>R_*$.
Theorem~\ref{thm:small-gaps} now gives a $\gamma$-pseudo insertor
$Z=\{x,y_1,\ldots,y_t\}$ in $P$, with its defining labelling.

In the target law $\I_t$, the comparison probabilities are
$i/(t+1)$, $1\le i\le t$. Since the TV error is at most $\gamma<\ep$, the selected vertices all belong to $P_0$, so
$$\TV\bigl(\L_{P_0;Z},\I_t\bigr)\le\gamma<\ep.$$
\end{proof}

The same reduction gives a short proof of the $1/e$ bound.
\begin{proof}[Proof that $w(P)\to\oo$ implies $\de(P)\ge1/e-o(1)$]
Suppose instead that $w(P)\to\oo$ and $\de(P)\le1/e-\ep$ for some
fixed $\ep>0$. As in the preceding proof, adjoin one universal minimum
and maximum. Lemma~\ref{cp:lem:localization} and the window bounds then
give $G=o(R)$ and $R\to\oo$, while $\de(P)$ is unchanged.

Choose $x,y$ with $d(x,y)=R$ and $\Delta(x,y)\le R$, and list the
vertices from $x$ to $y$ in nondecreasing height order as
$u_0=x,\ldots,u_m=y$, interchanging $x,y$ if necessary.
Put $\Delta_i=\Delta(u_{i-1},u_i)$ and $d_i=d(u_{i-1},u_i)$.
Equation~\eqref{eq:sub-add} gives
\[
 R^2\le\sum_{i=1}^m d_i^2,\qquad
 \sum_{i=1}^m\Delta_i^2\le G\sum_{i=1}^m\Delta_i
 =G\Delta(x,y)\le GR.
\]
Consequently some consecutive pair $u,v$ satisfies
$\Delta(u,v)/d(u,v)\le\sqrt{G/R}=o(1)$.

For the log-concave variable $V=Z_u-Z_v$, both sides of its mean have
probability at least $1/e$. Indeed, its survival function $S$ is
log-concave, so Jensen gives
$\log S(\EE V)\ge\EE\log S(V)=-1$; apply the same argument to $-V$.
Moving the threshold from $\EE V$ to zero and using the density bound
in~\eqref{eq:inputs}, we obtain
\[
 \de(P)\ge\min\{\PP(V<0),\PP(V>0)\}
 \ge\frac1e-\sqrt2\,\frac{\Delta(u,v)}{d(u,v)}
 =\frac1e-o(1),
\]
a contradiction.
\end{proof}

\section{Proof of Theorem~\ref{thm:small-gaps}}\label{sec:small-gap-proof}

\begin{proof}
It suffices to treat $0<\ep<\f{1}{2(t+1)}$. Fix $k\ge2$ and $t\ge1$.
We prove that there is $R_*=R_*(k,t,\ep)$ such that the second
alternative holds for $P_0$ with $G(P_0)<R(P_0)$, $R(P_0)\ge R_*$, and
no $\ep$-pseudo antichain of size $k$. 

\textbf{Stage 1: Normalize coordinates and count them locally.}
Writing $m=|P_0|$, form $P=C_{m^2}\op P_0\op C_{m^2}$ and put
$n=|P|=2m^2+m$. Clearly $w(P)=w(P_0)$, $G(P)=G(P_0)$, and $P$ has no $\ep$-pseudo antichain.

For distinct $u,v\in P$, put $D=f(v)-f(u)$. Conditional on the extension,
the rank-unit difference $Z_v-Z_u$ has mean $D$ and variance
$|D|(n+1-|D|)/(n+2)$. Thus
\begin{equation}\label{cp:eq:rankvariance}
 d(u,v)^2=\Var(D)+\EE\f{|D|(n+1-|D|)}{n+2}\le\Var(D)+\EE|D|.
\end{equation}
For $u,v \in P_0$, the law of $D$ and its mean are unchanged by padding,
while the conditional variance of $Z_u-Z_v$ at $s=|D|$ is $s-s(s+1)/(n+2)$, which increases with $n$. Thus $d_P(u,v)\ge d_{P_0}(u,v)$ and $\De_P(u,v) =\De_{P_0}(u,v)$, so $R:=R(P)\ge R(P_0)$. Hence it suffices to find an $\ep$-pseudo insertor in $P$.

For $u,v \nin P_0$, $D$ is a fixed nonzero integer, so $d(u,v)^2<|D|\le D^2=\De^2(u,v)$ and the pair cannot enter the maximum defining $R$. For $u,v\in P_0$, $|D|\le m$ gives $d(u,v)^2\le m^2+m$. For $u\in P_0,v\nin P_0$, $D$ has constant sign and ranges over an interval of at most $m$ values, so if this pair enters the maximum, then $\EE|D|=|\EE D|\le d(u,v)$, so $d(u,v)^2\le m^2+d(u,v)$ and $d(u,v)\le m+1$. Consequently $R\le m+1$, and a pair attaining $R>1$ has a vertex in $P_0$.

Choose a pair $x,y$ attaining $R$, with $x\in P_0$; then $\De(x,y)\le d(x,y)=R$. For each vertex $z\in P$, put
\begin{equation}\label{cp:eq:finite-coordinates}
 A_z=\f{Z_z-Z_x}{R}=b_z+W_z,\qquad
 b_z=\f{h(z)-h(x)}R,\qquad \EE W_z=0.
\end{equation}
We think of $b_z$ as the mean height of $z$ relative to the anchor $x$, and $W_z$ as its ``centered fluctuation,'' both in units of $R$. Write
\[
 d_0(s,t)=\|W_s-W_t\|_2=\f{d(s,t)}R,\qquad
 \ell_0(s,t)=|b_s-b_t|+d_0(s,t).
\]
In particular $b_x=W_x=0$, $|b_y|\le1$, and $d_0(x,y)=1$. The joint law of $(W_z)_{z\in P}$ is log-concave, since log-concavity is preserved by the affine maps.

Since $G(P)<R$, Lemma~\ref{cp:lem:scale} with $M=R$ gives
\begin{equation}\label{cp:fin:variance}
 \EE W_z^2=\f{d(z,x)^2}{R^2}
 \le 4\left(1+\f{\Delta(z,x)}R\right)=4(1+|b_z|).
\end{equation}
Similarly,
\begin{equation}\label{cp:fin:distance}
 d_0(s,t)=\f{d(s,t)}R
 \le\max\left\{1,\f{\Delta(s,t)}R\right\}
 =\max\{1,|b_s-b_t|\}.
\end{equation}
The metric $\ell_0$ has a uniform doubling constant $N_{k,\ep}$ by Lemma~\ref{cp:lem:doubling}. 

For an interval $I$ (with each endpoint either included or excluded), write $|I|$ for its length and put
$$N(w;I)=\f1R\#\{z:b_z+w_z\in I\}.$$
By Lemma~\ref{cp:lem:scale}, every interval $I$ of length $1$ contains at most $9R$ of the mean heights $b_z$. We show that, for every fixed $T,\ze,q>0$, for sufficiently large $R$,
\begin{equation}\label{cp:fin:counts}
 \PP\left(\sup_{I\se[-T,T]}|N(W;I)-|I||>\ze\right)<q.
\end{equation}
\bpr{Proof of~\eqref{cp:fin:counts}} Conditional on the linear extension $f$, let $f(x)=r$. Exchangeability of the $n+1$ uniform spacings gives, for fixed $u>0$ and sufficiently large $R$,
\[
 \#\{z:0<A_z\le u\}\ \stackrel{\mathrm{law}}=\
 \min\{\operatorname{Bin}(n,uR/(n+1)),\,n-r\}.
\]
Indeed, consecutive spacings to the right of $U_{(r)}$ have the law
of an initial segment of the spacings starting at zero; their partial
sums count uniform sample points, stopped after $n-r$ points.
Since $x\in P_0$ and $R\le m+1$,
\[
 \min\{r-1,n-r\}\ge m^2,\qquad m\ge R-1\ge R/2\quad(R\ge2).
\]
The probability of truncation is at most $uR/m^2\le4u/R$ by
Markov's inequality. The binomial mean differs from $uR$ by
$uR/(n+1)$, and its variance is at most $uR$. Once $u/(n+1)\le\zeta/2$,
Chebyshev's inequality therefore gives
\[
 \PP\left(\left|\f1R\#\{z:0<A_z\le u\}-u\right|>\zeta\right)
 \le\f{4u}{R}+\f{4u}{\zeta^2R}.
\]
The left side of the anchor is identical, with truncation at $r-1$;
the anchor contributes only $1/R$. Subtraction gives the assertion
for each fixed interval. For uniformity, choose a finite grid in
$[-T,T]$ of mesh at most $\zeta/16$, including both endpoints.
The union bound makes every grid-interval error at most $\zeta/2$
outside an event of probability $q$. Any interval is sandwiched
between grid intervals whose lengths differ from its length by at
most four mesh widths; the inner interval may be empty. Monotonicity
of counts now gives an error at most $\zeta/2+\zeta/4<\zeta$.
Open and closed endpoint conventions are covered by the same
sandwich, including the anchor. This proves~\eqref{cp:fin:counts}.\epr

\textbf{Stage 2: Choose the reference vertex.}
Fix $H>0$ and $0<\eta<1$, and put $s=\eta/2$ and
$I_h=(h-s,h+s]$. We choose a vertex $p$ near height zero and finitely
many measurements that distinguish vertices far from $p$. Stage 3
will find a vertex close to $p$ in each $I_h$, $|h|\le H$.

For $g\in\mathcal G:=\{0\le g\le2,\ \EE g=1\}$, put
$u_g(z)=\EE[gW_z]$. Since $\EE W_z=0$ and $\|g-1\|_2\le1$,
Cauchy--Schwarz and~\eqref{cp:fin:variance} give
\begin{equation}\label{cp:fin:weighted-growth}
 |u_g(z)|=|\EE[(g-1)W_z]|
 \le\|g-1\|_2\|W_z\|_2\le2\sqrt{1+|b_z|}.
\end{equation}
Likewise,
\begin{equation}\label{cp:fin:weighted-lipschitz}
 |u_g(v)-u_g(w)|
 =|\EE[(g-1)(W_v-W_w)]|
 \le\|g-1\|_2\|W_v-W_w\|_2\le d_0(v,w).
\end{equation}

For centered log-concave $V$, H\"older's inequality and
\eqref{eq:logconcave-tails} give
\[
 \EE|V|^2\le(\EE|V|)^{2/3}(\EE|V|^4)^{1/3}
 \le(24e)^{1/3}(\EE|V|)^{2/3}(\EE|V|^2)^{2/3}.
\]
Thus $\EE|V|\ge c_0\|V\|_2$, where $c_0=(24e)^{-1/2}$;
the zero-variance case is immediate. Put $r_0=c_0\eta/4$.
Choose an $\ell_0$-net $v_1,\ldots,v_q$ at resolution $r_0$ for
$|b_z|\le H+1$. By~\eqref{cp:fin:distance}, this window lies within
$\ell_0$-distance $2(H+1)$ of $x$. Doubling therefore gives
$q\le Q=Q(k,\ep,H,\eta)$, with $Q$ a fixed positive integer.

Let $\mu_0$ be uniform on $\{z:b_z\in I_0\}$, which contains $x$.
Doubling covers this set by at most $K_0=K_0(k,\ep,\eta)$ $\ell_0$-balls of radius
$r_0/2$. Assign each vertex to one containing ball. The resulting
clusters have $d_0$-diameter at most $r_0$, and a largest cluster
$\mathcal C_*$ has $\mu_0$-mass at least $\alpha_0:=1/K_0$.
Choose $p\in\mathcal C_*$. In particular, $|b_p|\le s<\eta$.

For each representative $v_i$, set
$$g_i=1+\tfrac12(\operatorname{sgn}(W_{v_i}-W_p)-\EE\operatorname{sgn}(W_{v_i}-W_p)),\wsp t_i(z)=u_{g_i}(z)-u_{g_i}(p),$$

where $\operatorname{sgn}(0)=0$. Since the sign lies in $[-1,1]$,
$0\le g_i\le2$ and $\EE g_i=1$. Centering $W_z-W_p$ gives
\[
 t_i(z)=\tfrac12\EE[\operatorname{sgn}(W_{v_i}-W_p)(W_z-W_p)],
 \qquad |t_i(z)|\le\tfrac12 d_0(z,p).
\]
At $v_i$, $t_i(v_i)=\frac12\EE|W_{v_i}-W_p|\ge\frac{c_0}{2}d_0(v_i,p)$;
moving to $z$ changes this value by at most $d_0(z,v_i)/2$.
If $|b_z|\le H+1$ and $d_0(z,p)\ge\eta$, choose $v_i$ with
$d_0(z,v_i)\le\ell_0(z,v_i)\le r_0$. Then
\[
 t_i(z)\ge\tfrac12\bigl(c_0d_0(v_i,p)-d_0(z,v_i)\bigr)
 \ge\tfrac12\bigl(c_0(\eta-r_0)-r_0\bigr)\ge r_0,
\]
because $c_0\eta=4r_0$ and $c_0\le1$. Consequently
\begin{equation}\label{cp:fin:separation}
 \begin{gathered}
 z\in\mathcal C_*\ \Longrightarrow\ \max_i|t_i(z)|\le r_0/2,
 \qquad \mu_0(\mathcal C_*)\ge\alpha_0,\\
 d_0(z,p)\ge\eta\ \Longrightarrow\ \max_i t_i(z)\ge r_0
 \qquad(|b_z|\le H+1).
 \end{gathered}
\end{equation}
Thus a vertex with $b_z\in I_h$ and all $|t_i(z)|<r_0$ is close enough to $p$.
We prove its existence by counting crossings and then coloring vertices.

\textbf{Stage 3: Weak uniformity of interval counts.} We prove that for every fixed $T,\zeta>0$, for sufficiently large $R$, each $g\in \mathcal G$ and $I\se [-T, T]$ satisfies
\begin{equation}\label{cp:fin:mean-density}
\ll|N(u_g;I)-|I|\rr|\le\ze.
\end{equation}
This includes $N(0;I)$, since the constant weight $g=1$ has $u_1=0$.
The first step is to find a configuration close to $u_g$ where the
random interval counts hold.

Rescale coordinate $z$ by $1+|b_z|$ and use the norm
$$\cpnorm w_*^4=\f1R\sum_z\f{|w_z|^4}{(1+|b_z|)^4}.$$

The fourth-moment bound following~\eqref{eq:logconcave-tails}, together
with~\eqref{cp:fin:variance}, gives
$$\EE|W_z|^4\le24e(\EE W_z^2)^2\le384e(1+|b_z|)^2.$$

For each integer $j\ge0$, at most $18R$ vertices satisfy
$j\le|b_z|<j+1$, counting multiplicities of equal mean heights.
Summing over these bands gives the two bounds needed below:
\[
 \begin{aligned}
 \f1R\sum_z(1+|b_z|)^{-2}
 &\le18\sum_{j\ge0}(j+1)^{-2}
 \le18\left(1+\int_1^\infty t^{-2}\,dt\right)=36,\\
 \f1R\sum_z(1+|b_z|)^{-4}
 &\le18\sum_{j\ge0}(j+1)^{-4}
 \le18\left(1+\int_1^\infty t^{-4}\,dt\right)=24.
 \end{aligned}
\]
Thus linearity of expectation gives
\begin{equation}\label{cp:fin:fourth-moment}
 \EE\cpnorm W_*^4\le\f{384e}{R}\sum_z(1+|b_z|)^{-2}
 \le384e\cdot36=13824e.
\end{equation}
For $T\ge1$, the same band count gives
\begin{equation}\label{cp:fin:fourth-tail}
 \f1R\sum_{|b_z|>T}\f{\EE|W_z|^4}{(1+|b_z|)^4}
 \le6912e\sum_{j\ge\lfloor T\rfloor}(j+1)^{-2}
 \le\f{6912e}{\lfloor T\rfloor}\le\f{13824e}{T}.
\end{equation}
Here the tail sum is at most $\int_{\lfloor T\rfloor}^\infty t^{-2}\,dt$
and $\lfloor T\rfloor\ge T/2$.

For each $r>0$ and $0<\alpha<1$, we claim that uniformly boundedly many
$\cpnorm{\cdot}_*$-balls of radius $r$ cover probability at least
$1-\alpha$; only their centers depend on the poset.

For $T\ge1$, \eqref{cp:fin:distance} places the vertices with $|b_z|\le T$
within $\ell_0$-distance $2T$ of $x$. Doubling gives an
$\ell_0$-net at resolution $\xi$ with at most
$K_1=K_1(T,\xi,k,\ep)$ representatives in this height window. Choose $\rh(z)$ with
$\ell_0(z,\rh(z))\le\xi$.
Define $V_z=W_{\rh(z)}$ on this height window and $V_z=0$ outside it.
Since $W_z-W_{\rh(z)}$ is centered and log-concave,
\[
 \EE|W_z-W_{\rh(z)}|^4
 \le24e\,d_0(z,\rh(z))^4\le24e\,\xi^4.
\]
The second band sum bounds the contribution inside the window;
\eqref{cp:fin:fourth-tail} bounds the contribution outside. Thus
\[
 \EE\cpnorm{W-V}_*^4
 \le576e\,\xi^4+\f{13824e}{T}.
\]
Set
\[
 \theta=\f{\alpha r^4}{324},\qquad
 T=1+\f{27648e}{\theta},\qquad
 \xi=\left(\f{\theta}{1152e}\right)^{1/4}.
\]
The last expectation is less than $\theta$, so Markov's inequality
gives $\PP(\cpnorm{W-V}_*\ge r/3)<81\theta/r^4=\alpha/4$.
Each representative has variance at most $4(1+T)$ by
\eqref{cp:fin:variance}. Take
$M=\sqrt{8K_1(1+T)/\alpha}$ and $\delta=r/8$.
Chebyshev's inequality and the union bound give probability at most
$4K_1(1+T)/M^2=\alpha/2$ that any representative lies outside $[-M,M]$.
When all representative values lie in $[-M,M]$, round each to a
multiple of $\delta$ with error at most $\delta$, and assign its
rounded value to the same vertices as before, producing $\widetilde V$.
Keep $\widetilde V_z=0$ outside the window. Then
\[
 \cpnorm{V-\widetilde V}_*^4
 \le \delta^4\,\f1R\sum_z(1+|b_z|)^{-4}
 \le24(r/8)^4<(r/3)^4.
\]
There are at most $(2\lceil M/\delta\rceil+3)^{K_1}$ rounded vectors.
Except on an event of probability less than $\alpha/4+\alpha/2<\alpha$,
one is within $2r/3$ of $W$, proving the cover claim.

We next show that
\begin{equation}\label{cp:fin:mean-ball}
\text{for each }r>0\text{ there is }\tau>0\text{ such that }\PP(\cpnorm{W-u_g}_*<r)\ge\tau\text{ for all }g\in\mathcal G.
\end{equation}

\bpr[Proof of~\eqref{cp:fin:mean-ball}]
Put $B_*=(13824e)^{1/4}$ and
$\alpha=\min\{1/16,r^2/(6400B_*^2)\}$.
The cover gives at most $K=K(r,\alpha,k,\ep)$ closed balls of
radius $r/4$ covering probability at least $1-\alpha$.
Discard balls of probability below $\alpha/K$. Their union
has probability at most $\alpha$, so the retained balls cover
probability at least $1-2\alpha$.

Define
\[
 B=\{w:\cpnorm w_*\le r/2\},\qquad
 \Phi(c)=\PP(W\in c+B),\qquad A=\{c:\Phi(c)\ge\alpha/K\}.
\]
If $w$ belongs to a retained ball, that entire ball lies in $w+B$.
Thus $w\in A$ and $\PP(W\in A)\ge1-2\alpha$.

Let $f_W$ be the density of $W$ on its affine support $S$. Then
\[
 \Phi(c)=\int_S f_W(w)\mathbf1_B(w-c)\,dw.
\]
The integrand is log-concave jointly in $(w,c)$, because $f_W$ is
log-concave and $\{(w,c):w-c\in B\}$ is convex. Integrating out $w$
preserves log-concavity by \cite[Theorem 5.1]{LV}. Thus $\log\Phi$
is concave wherever $\Phi>0$.

Fix $g\in\mathcal G$. Put $D=\{W\notin A\}$,
$\delta=\EE[g\mathbf1_D]\le2\PP(D)\le4\alpha$, and
$m=1-\delta\ge1/2$, and define
\[
 v=\frac{\EE[gW\mathbf1_{D^c}]}m.
\]
Jensen's inequality for $\log\Phi$, under the probability measure
with density $g\mathbf1_{D^c}/m$, gives
\[
 \log\Phi(v)\ge\f1m\EE[g\mathbf1_{D^c}\log\Phi(W)]
 \ge\log(\alpha/K).
\]
Thus $\Phi(v)\ge\alpha/K$.

The fourth-moment bound \eqref{cp:fin:fourth-moment} gives
$\EE\cpnorm W_*^2\le B_*^2$ and $\EE\cpnorm W_*\le B_*$.
Consequently $\cpnorm v_*\le2B_*/m\le4B_*$.
Since $u_g=mv+\EE[gW\mathbf1_D]$, Cauchy--Schwarz gives
\[
 \begin{aligned}
 \cpnorm{u_g-v}_*
 &\le\delta\cpnorm v_*+\EE[g\cpnorm W_*\mathbf1_D]\\
 &\le16B_*\alpha+
  2(\EE\cpnorm W_*^2)^{1/2}\PP(D)^{1/2}\\
 &\le16B_*\alpha+2B_*\sqrt{2\alpha}
 \le20B_*\sqrt\alpha\le r/4<r/2.
 \end{aligned}
\]
Here we used $\alpha<1$ and $\alpha\le r^2/(6400B_*^2)$.
Hence $\cpnorm{W-v}_*\le r/2$ implies $\cpnorm{W-u_g}_*<r$,
and
\[
 \PP(\cpnorm{W-u_g}_*<r)\ge\Phi(v)\ge\alpha/K.
\]
This proves~\eqref{cp:fin:mean-ball} with $\tau=\alpha/K$.\epr

\bpr[Proof of~\eqref{cp:fin:mean-density}] Fix $T,\zeta>0$ and put $B=64+4(T+1)$. For $u>B$,
$2\sqrt{1+u}\le u/3$ and $T+1\le u/4$, so
\[
 u-2\sqrt{1+u}-(T+1)\ge5u/12\ge(1+u)/4.
\]
By~\eqref{cp:fin:weighted-growth}, a vertex with $|b_z|>B$ cannot
have $b_z+u_g(z)\in[-T-1,T+1]$.
If $b_z+w_z$ lies there, $e_z:=w_z-u_g(z)$ must satisfy
$|e_z|\ge(1+|b_z|)/4$. For $0<\delta<1$, the normalized number of
these far vertices, together with those having $|b_z|\le B$ and
$|e_z|>\delta$, is at most
\[
 \left(256+\f{(1+B)^4}{\delta^4}\right)\cpnorm e_*^4.
\]
Indeed, each far vertex contributes at least $1/(256R)$ to the
fourth power of the norm, and each of the other exceptional vertices
contributes more than $\delta^4/[R(1+B)^4]$.

Set
\[
 \delta=\min\{1/2,\zeta/16\},\qquad
 r=\left(\f{\zeta}{4(256+(1+B)^4/\delta^4)}\right)^{1/4}.
\]
The exceptional count is then at most $\zeta/4$ whenever $\cpnorm e_*<r$.
Obtain $\tau>0$ from~\eqref{cp:fin:mean-ball}. By~\eqref{cp:fin:counts},
the event that every interval in $[-T-1,T+1]$ has counting error at
most $\zeta/4$ has probability greater than $1-\tau/2$ for large $R$.
For each $g$, it therefore meets $\{\cpnorm{W-u_g}_*<r\}$, whose
probability is at least $\tau$. Choose $w$ in the intersection.

For $I\subseteq[-T,T]$, enlarge each endpoint by $\delta$ to obtain
$I^+$, and shrink by $\delta$ to obtain $I^-$; use closed outer and
open inner intervals, with $I^-=\varnothing$ if necessary.
Every nonexceptional vertex contributing to these counts has
$|b_z|\le B$ and displacement at most $\delta$. Hence
\[
 N(w;I^-)-\zeta/4\le N(u_g;I)\le N(w;I^+)+\zeta/4.
\]
The interval-length error is at most $2\delta$ on either side
(if $I^-$ is empty, $|I|\le2\delta$). Hence
$|N(u_g;I)-|I||\le2\delta+\zeta/2\le5\zeta/8<\zeta$, proving~\eqref{cp:fin:mean-density}
\epr

\textbf{Stage 4: Uniform boundedness of the weighted means.}
Assign each $z\in P$ a (deterministic) ``starting position'' $x_z$ and
``velocity'' $v_z$, so that its position at time $t$ is $x_z+t v_z$.
Assume $|v_z|\le M$ with $M>0$, fix weights $0\le w_z\le1$,
and let $\tau>0$. Suppose, for some $c\ge0$, that the normalized
weighted counts satisfy

\begin{equation}\label{cp:cross:hypothesis}
 \left|\f1R\sum_z w_z\mathbf1_{\{x_z+t v_z\in I\}}-c|I|\right|
 \le\ep\qquad(t\in\{0,\tau\})
\end{equation}
for every interval $I$ in some fixed window. Put $E=M(4c\tau M+2\ep)$. We show that, for intervals $I$, $J$ whose $\tau M$-neighborhoods lie within that window, 
\begin{equation}\label{cp:cross:comparison}
 \left|\f1R\sum_{x_z\in I}w_zv_z
       -\f{|I|}{|J|}\f1R\sum_{x_z\in J}w_zv_z\right|
 \le E\left(1+\f{|I|}{|J|}\right)
       +\f{2\varepsilon|I|}{\tau}.
\end{equation}
\bpr[Proof of~\eqref{cp:cross:comparison}]
For each point, count movements during $[0,\tau]$ from at or below
it to above it with weight $w_z/R$, and movements in the reverse
direction with weight $-w_z/R$.
For $a<b$ in the window, the signed crossing count at $b$ minus
that at $a$ is
\[
 \f1R\sum_z w_z
 \bigl(\mathbf1_{\{x_z\in(a,b]\}}
       -\mathbf1_{\{x_z+\tau v_z\in(a,b]\}}\bigr).
\]
Each of these two interval counts differs from $c(b-a)$ by at most
$\varepsilon$. Thus signed crossing counts at any two points
differ by at most $2\varepsilon$.

For a positive-length interval $I$, partition it at all trajectory
endpoints $x_z,x_z+\tau v_z$ lying inside it. The signed crossing
count is constant on each open piece. Average these counts with
weights equal to the piece lengths divided by $|I|$.
Exchanging the finite sums over pieces and vertices shows that this
average equals the normalized weighted sum of signed trajectory
lengths inside $I$, divided by $|I|$.
To that sum, a trajectory wholly inside contributes exactly
$\tau w_zv_z/R$. Discrepancies from counting the full trajectory of
each vertex initially in $I$ can arise only within $\tau M$ of the
endpoints of $I$. The two endpoint strips have total normalized weight at
most $4c\tau M+2\varepsilon$, and each trajectory has length at most
$\tau M$. 

The average crossing count of $I$ therefore differs from
$\frac{\tau}{R|I|}\sum_{x_z\in I}w_zv_z$ by at most $\tau E/|I|$.
Vertices at endpoints are included in these strips, so either
endpoint convention is allowed. The averages for any two intervals
differ by at most $2\varepsilon$, since all the signed crossing counts do.
For positive-length intervals $I,J$ whose $\tau M$-neighborhoods lie
in that window, the triangle inequality gives a total error
$\tau E/|I|+2\varepsilon+\tau E/|J|$ between the corresponding
velocity sums scaled by $\tau/(R|I|)$ and $\tau/(R|J|)$.
Multiplying by $|I|/\tau$ gives \eqref{cp:cross:comparison}.\epr

Now for every fixed $T>0$, we claim that, for sufficiently large $R$,
\begin{equation}\label{cp:fin:bounded}
 g\in\mathcal G,\ |b_z|\le T \imrsp |u_g(z)|\le3.
\end{equation}
\bpr[Proof of~\eqref{cp:fin:bounded}] Put $T_1=T+3$. We shall apply \eqref{cp:cross:comparison} to $\{z\in P:|b_z|\le 68+4T_1\}$.
By~\eqref{cp:fin:weighted-growth}, discarded vertices cannot
enter $[-T_1,T_1]$ under $b_z+\lambda u_g(z)$, $0\le\lambda\le1$.
On the retained set take 
$$x_z=b_z,\wsp v_z=u_g(z),\wsp w_z=1,\wsp M=2\s{69+4T_1}.$$
The weight $1+\la(g-1)$ belongs to
$\mathcal G$, so~\eqref{cp:fin:mean-density} supplies~\eqref{cp:cross:hypothesis} with $c=1$ and any prescribed $\ep>0$.

Write $J_h=(h-1/4,h+1/4]$.
Set $\tau=1/(128M^2)$ and $\varepsilon=\tau/32$.
Since $M\ge1$, we have $\tau\le1$, $\tau M\le1$,
$4M\le1/\tau$, and $\varepsilon\le1/8$.
For $|h|\le T$, the $\tau M$-neighborhoods of $J_h$ and $J_0$ lie
in $[-T_1,T_1]$. Apply~\eqref{cp:cross:comparison} with
$I=J_h$ and $J=J_0$, both of length $1/2$, to obtain
$$\ll|\f1R\sum_{b_z\in J_h}u_g(z)
       -\f1R\sum_{b_z\in J_0}u_g(z)\rr|\le8M^2\tau+(4M+1/\tau)\ep\le\f1{16}+\f{2\ep}{\tau}=\f18.
$$

Since $u_g(x)=0$, \eqref{cp:fin:distance} and
\eqref{cp:fin:weighted-lipschitz} give $|u_g(z)|\le1$ on $J_0$.
Thus the normalized sum over $J_0$ has absolute value at most
$N(0;J_0)\le1/2+\ep$, and the sum over $J_h$ has absolute
value at most $1/2+\ep+1/8$.
For a vertex $z$ with $|b_z|\le T$,
every vertex in $J_{b_z}$ has measurement within $1$ of $u_g(z)$,
again by these two inequalities. Since
$N(0;J_{b_z})\ge1/2-\ep>0$, it follows that

$$|u_g(z)|\le1+\f{\ll|\tfrac{1}{R}\sum_{b_y\in J_{b_z}}u_g(y)\rr|}{N(0;J_{b_z})}\le1+\f{1/2+\ep+1/8}{1/2-\ep}\le1+\f{3/4}{3/8}=3.$$
This proves~\eqref{cp:fin:bounded}.\epr

\textbf{Stage 5: Use finite coin flips to detect copies.} 

Fix $B>0$ (to be defined below, and put, for $z\in P$ with $|b_z|\le B$ and $1\le i\le q$,
$$p_i(z):=\f{u_{g_i}(z)+3}{6}$$

For fixed weights $0\le w_z\le1$ and $a\in \RR^q$, put
\[
 x_z(a)=b_z+\sum_{i=1}^q a_i p_i(z),\qquad
 N_w(a;I)=\f1R\sum_z w_z\mathbf1_{\{x_z(a)\in I\}}.
\]
where the second sum runs only over $z$ for which $x_z(a)$ is defined.

Suppose, for $0<A\le1$, $0\le c\le2$, $0<\varepsilon\le1$, and
$L\ge H_1\ge1$, that
\[
 |N_w(a;I)-c|I||\le\varepsilon
 \quad(\|a\|_1\le A,\ I\subseteq[-L-3,L+3]).
\]
Assume also that $N_1(0;I)\le|I|+1$ in this window; this will follow
from the initial unweighted counting estimate.
Take $0<\tau<A$, and for fixed $i$ replace $w_z$ by either
$w'_z=w_zp_i(z)$ or $w'_z=w_z(1-p_i(z))$. By~\eqref{cp:fin:bounded}, for sufficiently large $R$, $p_i(z)\in [0,1]$. 

For $\|a\|_1\le A-\tau$, if $e_i$ is the indicator vector for the $i$th coordinate, moving with velocity $p_i$ changes $a$ to
$a+\tau e_i$; moving with velocity $1-p_i$ changes it to
$a-\tau e_i$ and translates all positions by $\tau$.
Both movements satisfy the assumed count bound. The extra margin
of $3$ covers these translations and the endpoint strips. 

Apply~\eqref{cp:cross:comparison} with $M=1$, $J=(-L,L]$, and
$I\subseteq[-H_1,H_1]$. Here $E\le8\tau+2\varepsilon$.
Set $c'=N_{w'}(0;J)/(2L)$. Because $w'\le1$ and $L\ge1$,
$0\le c'\le(2L+1)/(2L)\le2$.
Also $|x_z(a)-b_z|\le\|a\|_1\le1$; changing $a$ to zero can
alter membership in $J$ only in its two endpoint strips of length
$2$. Since $w'\le w$, their total normalized weight is at most
$4c+2\varepsilon$. Hence
$$|N_{w'}(a;I)-c'|I||\le(8\tau+2\ep)\ll(1+\f{H_1}{L}\rr)+\f{4H_1\ep}{\tau}+\f{H_1}{L}(4c+2\ep)\le16\tau+4\ep+\f{4H_1\ep}{\tau}+\f{10H_1}{L}.$$

The last inequality uses $H_1/L\le1$, $c\le2$, and $\ep\le1$.
Thus the new counts are uniform, with a density $c'$ independent of
the movement parameter $a$.

For any fixed number $D\ge1$ of rounds, target window $[-H_D,H_D]$
with $H_D\ge1$, and error $0<\varepsilon_D\le1$, the preceding estimate can
be iterated uniformly over every sequence of color choices.
Working backwards for $j=D,D-1,\ldots,1$, define
\[
 \tau_j=\f{\varepsilon_j}{64D},\qquad
 L_j=\f{40H_j}{\varepsilon_j},\qquad
 H_{j-1}=L_j+3,\qquad
 \varepsilon_{j-1}=\f{\varepsilon_j\tau_j}{16H_j}.
\]
These definitions preserve $H_j\ge1$ and $0<\varepsilon_j\le1$,
so $\tau_j\le1/(64D)\le H_j$. The one-coloring error is at most
$\varepsilon_j$, because
\[
 \begin{gathered}
 16\tau_j=\f{\varepsilon_j}{4D}\le\f{\varepsilon_j}{4},\qquad
 \f{10H_j}{L_j}=\f{\varepsilon_j}{4},\\
 \left(4+\f{4H_j}{\tau_j}\right)\varepsilon_{j-1}
 =\f{\varepsilon_j}{4}\left(1+\f{\tau_j}{H_j}\right)
 \le\f{\varepsilon_j}{2}.
 \end{gathered}
\]
Start with $w=1$, density $c=1$, and counting
error $\varepsilon_0$ on $[-H_0,H_0]$ for $\|a\|_1\le1$.
After round $j$, the allowed parameter radius is
$1-\sum_{r=1}^j\tau_r\ge1-D/(64D)=63/64$; in particular, every
step has $\tau_j<A$. Each new density lies in $[0,2]$,
as just proved; the required unweighted upper bound holds throughout
because all reference windows lie in the initial window and
$\varepsilon_0\le1$. Induction therefore gives, for every resulting
weight $w$ and some $c_w\in[0,2]$,
\begin{equation}\label{cp:cross:colors}
 |N_w(0;I)-c_w|I||\le\varepsilon_D
 \qquad(I\subseteq[-H_D,H_D]).
\end{equation}

Let $\alpha=s\alpha_0$, set
\[
 \gamma=\f{\alpha}{8(2s+1)},\qquad
 n_*=\left\lceil\f{144Q}{\gamma r_0^2}\right\rceil,
\]
and take $D=n_*q$, $H_D=H+1$, and
$\varepsilon_D=\alpha/(8\cdot2^{n_*Q})$ in the finite construction
above.

Let $H_0$ and $\varepsilon_0$ be the resulting initial parameters, and let $B=64+4(H_0+2)$.

For $\|a\|_1\le1$, the weight
$g_a=1+\frac16\sum_i a_i(g_i-1)$ lies in $\mathcal G$, because
$|g_a-1|\le1/6$ and $\EE g_a=1$. Moreover,
$$b_z+\sum_i a_i p_i(z)=b_z+u_{g_a}(z)+\tfrac12\sum_i a_i.$$
The last translation has magnitude at most $1/2$. The cutoff
calculation in the proof of~\eqref{cp:fin:mean-density}, with
$T=H_0+1$, excludes discarded vertices even after this translation.
Thus
\eqref{cp:fin:mean-density}, applied on $[-H_0-1,H_0+1]$ with error
$\varepsilon_0$, supplies all the initial counting hypotheses for
\eqref{cp:cross:colors}.

For each vertex, flip $n_*$ coins of success probability $p_i(z)$ for
each $i$, all $n_*q$ coins independent. Let $\widehat p_i(z)$ be the success frequency
and put $\widehat u_i(z)=6\widehat p_i(z)-3$. The expectation of
$\widehat u_i(z)$ is $u_{g_i}(z)$, and
\[
 \Var(\widehat u_i(z))=36p_i(z)(1-p_i(z))/n_*\le9/n_*.
\]
Chebyshev's inequality and the union bound give
\[
 \PP\left(\max_i|\widehat u_i(z)-u_{g_i}(z)|\ge r_0/4\right)
 \le\f{144q}{n_*r_0^2}\le\gamma.
\]
Accept the outcomes if
$\max_i|\widehat u_i(z)-u_{g_i}(p)|<3r_0/4$.
By~\eqref{cp:fin:separation}, a vertex in $\mathcal C_*$ is accepted
with probability at least $1-\gamma$. A vertex in $|b_z|\le H+1$
with $d_0(z,p)\ge\eta$ is accepted with probability at most $\gamma$:
on the complementary event in the last display, some measured
difference is greater than $r_0-r_0/4=3r_0/4$.

Write $F(h)$ for the expected number of accepted vertices with
$b_z\in I_h$, divided by $R$. The acceptance rule specifies a fixed
subset of the $2^{n_*q}$ outcome sequences, the same for every vertex.
For each sequence, \eqref{cp:cross:colors} compares its expected
counts in $I_h$ and $I_0$ with error at most $2\varepsilon_D$, since
both intervals have length $2s$. Summing gives
\[
 |F(h)-F(0)|\le2^{n_*q}\cdot2\varepsilon_D\le\alpha/4
 \qquad(|h|\le H).
\]
Increase $R$ so that~\eqref{cp:fin:mean-density} also has error at
most $s<1$ on these intervals. Then
$N(0;I_0)\ge s$, so $|\mathcal C_*|/R\ge\alpha_0s=\alpha$,
and $N(0;I_h)\le2s+1$. Since $\gamma\le1/8<1/4$,
\[
 F(h)\ge F(0)-\alpha/4
 \ge(1-\gamma)\alpha-\alpha/4\ge\alpha/2.
\]
If every vertex in $I_h$ had $d_0(z,p)\ge\eta$, the acceptance
bound would instead give
$F(h)\le\gamma N(0;I_h)\le\gamma(2s+1)=\alpha/8$,
a contradiction. We have proved
\begin{equation}\label{cp:eq:replicas}
 \text{for every }h\in[-H,H]\text{ there is }z\in P
 \text{ with }|b_z-h|<\eta,\quad\|W_z-W_p\|_2<\eta.
\end{equation}
Indeed $b_z\in I_h$ gives $|b_z-h|\le s<\eta$.

\textbf{Stage 6: Construct the pseudo insertor.}
Put $H=1+2\s{t+1}$. Choose $\et>0$ sufficiently small (in terms of $t,\ep$) as specified below. Apply stages 2--5 with these parameters to obtain a vertex $p$ and the copies guaranteed by~\eqref{cp:eq:replicas}. Since $d_0(x,y)=1$, one $v\in\{x,y\}$ satisfies $d_0(v,p)\ge1/2$. Put $V=A_v-W_p$. Then~\eqref{cp:fin:distance} gives
$$
 \EE V=b_v,\qquad |b_v|\le1,\qquad
 1/2\le\sd(V)=d_0(v,p)\le\max\{1,|b_v-b_p|\}\le1+\et.
$$
Its nondegenerate log-concave law has a continuous distribution function.
By~\eqref{eq:inputs}, its density is bounded by
$\sqrt2/\sd(V)\le2\sqrt2$; put $K=2\sqrt2$.

Now for $1\le i\le t$, choose the quantile $q_i$ with
$\PP(V<q_i)=i/(t+1)$. Chebyshev's inequality and \eqref{eq:inputs} give
$$
 |q_i|\le1+(1+\et)\s{t+1}<H,\qquad
 q_{i+1}-q_i\ge\f{1}{K(t+1)}\quad(1\le i<t).
$$
Choose $z_i$ from~\eqref{cp:eq:replicas} at height $q_i$. Thus
$$
 A_v=W_p+V,\qquad A_{z_i}=W_p+q_i+e_i,
 \qquad e_i=b_{z_i}-q_i+W_{z_i}-W_p.
$$
$$\EE e_i^2=(b_{z_i}-q_i)^2+\|W_{z_i}-W_p\|_2^2\le2\et^2$$
The $z_i$'s are distinct if $2\et<\f{1}{K(t+1)}$, since
$|b_{z_i}-q_i|<\et$. Also $z_i\ne v$, since
$d_0(z_i,p)<\et<1/2\le d_0(v,p)$.

Put $Z=\{v,z_1,\ldots,z_t\}$. Couple its actual order to the ideal
order obtained by inserting $V$
among $q_1<\cds<q_t$; the latter has law $\I_t$.
If $h=\sqrt\et<\f{1}{2K(t+1)}$, the two orders agree whenever
$$
 \max_i|e_i|\le h\qquad\text{and}\qquad
 \min_i|V-q_i|>h.
$$
Indeed, the leaves remain in threshold order and each comparison with $v$ keeps its sign. The union bound, Chebyshev's inequality, and \eqref{eq:inputs} therefore give
\begin{equation}\label{eq:insertion-tv}
 \TV\bigl(\L_{P;Z},\I_t\bigr)
 \le \sum_{i=1}^t \PP(|e_i|>h)+\sum_{i=1}^t \PP(|V-q_i|\le h)
 \le\frac{2t\et^2}{h^2}+2Kth
 =2t\et+2Kt\sqrt\et.
\end{equation}
Note that this does not require independence between $V$ and the errors. Hence, for sufficiently small $\et>0$, we have
$$\TV\ll(\L_{P_0;Z},\I_t\rr)<\ep.$$
This proves the second alternative and completes the proof.
\end{proof}

\textit{Remark.} By calculating the dependence on $w(P)$ in the proof of Theorem~\ref{cor:kahn-saks}, it can be shown that
\begin{equation}
\de(P)\ge 1/2-O((\log\log\log\log\log w(P))^{-1/4})
\end{equation}
No attempt has been made at optimizing this.

\section*{AI Usage} ChatGPT 6 Astra was used primarily to aid in the discovery of the results here, and also lightly in the preparation of this manuscript. As usual, the author takes full responsibility for the contents of the paper.

\section*{Acknowledgements}
I would especially like to thank Swee Hong Chan, Greta Panova, and Igor Pak for discussions and invaluable input.

\end{document}